\documentclass[12pt]{amsart}

\usepackage{amsmath,amssymb,epic,lscape}

\newtheorem{theorem}{Theorem}[section]
\newtheorem{proposition}[theorem]{Proposition}
\newtheorem{lemma}[theorem]{Lemma}

\theoremstyle{definition}
\newtheorem{definition}[theorem]{Definition}

\theoremstyle{remark}

\numberwithin{equation}{section}
\providecommand{\bysame}{\leavevmode\hbox to3em{\hrulefill}\thinspace}

\def\DJ{{\hbox{D\kern-.8em\raise.15ex\hbox{--}\kern.35em}}}
\def\DJo{$\;$\kern-.4em
    \hbox{D\kern-.8em\raise.15ex\hbox{--}\kern.35em okovi\'c}}

\def\NSERC{Supported in part by an NSERC Discovery Grant.}

\def\al{{\alpha}}

\def\ve{{\varepsilon}}
\def\vf{{\varphi}}

\def\bZ{{\mbox{\bf Z}}}

\def\pE{{\mathcal E}}

\renewcommand{\subjclassname}{\textup{2000} Mathematics Subject
Classification }

\begin{document}

\title[Classification of near-normal sequences]
{Classification of near-normal sequences}

\author[D.\v{Z}. \DJ okovi\'{c}]
{Dragomir \v{Z}. \DJ okovi\'{c}}

\address{Department of Pure Mathematics, University of Waterloo,
Waterloo, Ontario, N2L 3G1, Canada}

\email{djokovic@uwaterloo.ca}

\thanks{\NSERC}


\date{}

\begin{abstract}
We introduce a canonical form for near-normal sequences $NN(n)$, 
and using it we enumerate the equivalence classes of such sequences
for even $n\le30$.
\end{abstract}

\maketitle
\subjclassname{ 05B20, 05B30 }
\vskip5mm

\section{Introduction}

Near-normal sequences were introduced by C.H. Yang in \cite{Y}.
They can be viewed as quadruples of binary sequences
$(A;B;C;D)$ where $A$ and $B$ have length $n+1$, while
$C$ and $D$ have length $n$, and $n$ has to be even.
By definition, the sequences $A=a_1,a_2,\ldots,a_{n+1}$ and
$B=b_1,b_2,\ldots,b_{n+1}$ are related by the equalities
$b_i=(-1)^{i-1}a_i$ for $1\le i\le n$ and $b_{n+1}=-a_{n+1}$.
Moreover it is required that the sum of
the non-periodic autocorrelation functions of the four sequences
be a delta function.

Examples of near-normal sequences are known for all even $n\le30$.
Due to the important role that these sequences play in various
combinatorial constructions such as that for $T$-sequences,
orthogonal designs, and Hadamard matrices \cite{HCD,SY,KSY},
it is of interest to classify the near-normal sequences of
small length. We shall give such classification for all even $n\le30$.
We have recently constructed \cite{DZ2} near-normal sequences
for $n=32$ and $n=34$.

In section \ref{baz} we recall from \cite{DZ1} the basic properties
of base sequences. In section \ref{sk-nor} we introduce two
equivalence relations for near-normal sequences: $BS$-equivalence
and $NN$-equivalence. The former is finer than the latter.
We also introduce the canonical form
for the $BS$-equivalence classes. By using this canonical form,
we are able to compute the representatives of the $BS$-equivalence
classes and then deduce the set of representatives for the
$NN$-equivalence classes. In section \ref{tab} we tabulate
our results giving the list of representatives of the
$NN$-equivalence classes. The representatives are written
in the encoded form used in our previous paper \cite{DZ1}.

\section{Base sequences} \label{baz}

We denote finite sequences of integers by capital letters. If, say,
$A$ is such a sequence of length $n$ then we denote its elements
by the corresponding lower case letters. Thus
$$ A=a_1,a_2,\ldots,a_n. $$
To this sequence we associate the polynomial
$$ A(x)=a_1+a_2x+\cdots+a_nx^{n-1} , $$
which we view as an element of the Laurent polynomial ring
$\bZ[x,x^{-1}]$. (As usual, $\bZ$ denotes the ring of integers.)
The non-periodic autocorrelation function $N_A$ of $A$ is defined by:
$$ N_A(i)=\sum_{j\in\bZ} a_ja_{i+j},\quad i\in\bZ, $$
where $a_k=0$ for $k<1$ and for $k>n$. Note that
$N_A(-i)=N_A(i)$ for all $i\in\bZ$ and $N_A(i)=0$ for $i\ge n$.
The norm of $A$ is the Laurent polynomial $N(A)=A(x)A(x^{-1})$.
We have
$$ N(A)=\sum_{i\in\bZ} N_A(i) x^i . $$
The negation, $-A$, of $A$ is the sequence
$$ -A=-a_1,-a_2,\ldots,-a_n. $$
The {\em reversed} sequence $A'$ and the {\em alternated} sequence
$A^*$ of the sequence $A$ are defined by
\begin{eqnarray*}
A' &=& a_n,a_{n-1},\ldots,a_1 \\
A^* &=& a_1,-a_2,a_3,-a_4,\ldots,(-1)^{n-1}a_n.
\end{eqnarray*}
Observe that $N(-A)=N(A')=N(A)$ and $N_{A^*}(i)=(-1)^i N_A(i)$
for all $i\in\bZ$.
By $A,B$ we denote the concatenation of the sequences $A$ and $B$.

A {\em binary sequence} is a sequence whose terms belong to the
set $\{\pm1\}$. When displaying such sequences, we shall often
write $+$ for $+1$ and $-$ for $-1$.
The {\em base sequences} consist of four 
binary sequences $(A;B;C;D)$, with $A$ and $B$ of length $m$ 
and $C$ and $D$ of length $n$, such that 
\begin{equation} \label{norma}
N(A)+N(B)+N(C)+N(D)=2(m+n).
\end{equation}
We denote by $BS(m,n)$ the set of such base sequences with
$m$ and $n$ fixed. 

We recall from \cite{DZ1} that two members of $BS(m,n)$ are said to be
{\em equivalent} if one can be transformed to the other by applying
a finite sequence of {\em elementary transformations}. The elementary
transformations of $(A;B;C;D)\in BS(m,n)$ are the following:

(i) Negate one of the four sequences $A;B;C;D$.

(ii) Reverse one of the sequences $A;B;C;D$.

(iii) Interchange two of the sequences $A;B;C;D$ of the same length.

(iv) Alternate all four sequences $A;B;C;D$.

One can view the equivalence classes in $BS(m,n)$ as orbits of an
abstract finite group $G$. We shall assume that $m\ne n$.
In that case $G$ has order $|G|=2^{11}$. To construct $G$, we start
with an elementary abelian group $E$ of order $2^8$ with generators
$\ve_i,\vf_i$, $i\in\{1,2,3,4\}$.
and an elementary abelian group $V$ of order $4$ with generators
$\sigma_1,\sigma_2$. Let $H$ be the semidirect product of $E$ and
$V$, with $V$ acting on $E$ so that $\sigma_1$ commutes with
$\ve_3,\ve_4,\vf_3,\vf_4$, and $\sigma_2$ commutes with
$\ve_1,\ve_2,\vf_1,\vf_2$, and
$$
\sigma_1\ve_1=\ve_2\sigma_1,\quad
\sigma_1\vf_1=\vf_2\sigma_1,\quad
\sigma_2\ve_3=\ve_4\sigma_2,\quad
\sigma_2\vf_3=\vf_4\sigma_2.
$$
Finally, we define $G$ as the semidirect product of $H$
and the group $Z_2$ of order 2 with generator $\psi$.
By definition, $\psi$ commutes with each $\ve_i$ and we have
$$ \psi\vf_i=\ve_i^{m-1}\vf_i\psi,\ i=1,2; \quad
\psi\vf_j=\ve_j^{n-1}\vf_j\psi,\ j=3,4. $$

The group $G$ acts on $BS(m,n)$ as follows.
If $S=(A;B;C;D)\in BS(m,n)$ then
\begin{eqnarray*}
&& \ve_1S=(-A;B;C;D),\quad \vf_1S=(A';B;C;D), \\
&& \ve_2S=(A;-B;C;D),\quad \vf_2S=(A;B';C;D), \\
&& \ve_3S=(A;B;-C;D),\quad \vf_3S=(A;B;C';D), \\
&& \ve_4S=(A;B;C;-D),\quad \vf_4S=(A;B;C;D'),
\end{eqnarray*}
and $\psi S=(A^*;B^*;C^*;D^*)$. It is easy to verify that the defining
relations of $G$ are satisfied by these transformations and so the
action of $G$ on $BS(m,n)$ is well defined. Consequently, the
following proposition holds.

\begin{proposition} \label{orbite}
If $m\ne n$, the orbits of $G$ in $BS(m,n)$ are the same as the
equivalence classes in $BS(m,n)$.
\end{proposition}

We need also the encoding scheme for the base sequences
$(A;B;C;D)$ in $BS(n+1,n)$ introduced in \cite{DZ1}.
We now recall that scheme.
We decompose the pair $(A;B)$ into quads
$$ \left[ \begin{array}{ll} a_i & a_{n+2-i} \\ 
b_i & b_{n+2-i} \end{array} \right],\quad i=1,2,\ldots,
\left[ \frac{n+1}{2} \right], $$ 
and, if $n=2m$ is even, the central column
$ \left[ \begin{array}{l} a_{m+1} \\ b_{m+1} \end{array} \right]. $
Up to equivalence of base sequences, we can assume that
the first quad of $(A;B)$ is 
$ \left[ \begin{array}{ll} + & + \\ + & - \end{array} \right]. $
We attach to this particular quad the label 0.
The other quads in $(A;B)$ and all the quads of the pair $(B;C)$,
shown with their labels, must be one of the following:
\begin{center}
\begin{eqnarray*}
1=\left[ \begin{array}{ll} + & + \\ + & + \end{array} \right],\quad 
2=\left[ \begin{array}{ll} + & + \\ - & - \end{array} \right],\quad 
3=\left[ \begin{array}{ll} - & + \\ - & + \end{array} \right],\quad 
4=\left[ \begin{array}{ll} + & - \\ - & + \end{array} \right], \\
5=\left[ \begin{array}{ll} - & + \\ + & - \end{array} \right],\quad 
6=\left[ \begin{array}{ll} + & - \\ + & - \end{array} \right],\quad 
7=\left[ \begin{array}{ll} - & - \\ + & + \end{array} \right],\quad 
8=\left[ \begin{array}{ll} - & - \\ - & - \end{array} \right].
\end{eqnarray*}
\end{center}
The central column is encoded as
$$
0=\left[ \begin{array}{l} + \\ + \end{array} \right],\quad
1=\left[ \begin{array}{l} + \\ - \end{array} \right],\quad
2=\left[ \begin{array}{l} - \\ + \end{array} \right],\quad
3=\left[ \begin{array}{l} - \\ - \end{array} \right].
$$

If $n=2m$ is even, the pair $(A;B)$ is encoded as the
sequence of digits $q_1q_2\ldots q_mq_{m+1}$, where $q_i$,
$1\le i\le m$, is the label of the $i$th quad and $q_{m+1}$
is the label of the central column. If $n=2m-1$ is odd,
then $(A;B)$ is encoded by $q_1q_2\ldots q_m$, where $q_i$
is the label of the  $i$th quad for each $i$.
We use the same recipe to encode the pair $(C;D)$.

As an example, the base sequences
\begin{eqnarray*}
A &=& +,+,+,+,-,-,+,-,+; \\
B &=& +,+,+,-,+,+,+,-,-; \\
C &=& +,+,-,-,+,-,-,+; \\
D &=& +,+,+,+,-,+,-,+; \\
\end{eqnarray*}
are encoded as $06142;\, 1675$.

\section{Near-normal sequences} \label{sk-nor}

{\em Near-normal sequences}, originally defined 
by C.H. Yang \cite{Y}, can be viewed as a special type of base 
sequences $(A;B;C;D)\in BS(n+1,n)$ (see \cite{KSY,DZ1}) with $n$ even,
namely such that $b_i=(-1)^{i-1} a_i$ for $1\le i\le n$.
Note that we also must have $b_{n+1}=-a_{n+1}$. Hence, the sequence
$B$ is uniquely determined by $A$, and we define $\al A=B$.
Note that also $\al B=A$.

We denote by $NN(n)$ the subset
of $BS(n+1,n)$ consisting of near-normal sequences.
It has been conjectured (Yang \cite{Y}) that
$NN(n)\ne\emptyset$ for all positive even $n$'s.
Yang's conjecture has been confirmed for all even $n\le34$ \cite{DZ2}.

We shall introduce two equivalence relations in $NN(n)$: 
$BS$-equivalence and $NN$-equivalence. The former is stronger than
the latter.

We say that two members of $NN(n)$ are {\em $BS$-equivalent} if they are
equivalent as base sequences in $BS(n+1,n)$.
One can enumerate the $BS$-equivalence classes by finding suitable
representatives of the classes.
For that purpose we introduce the concept of canonical form
for near-normal sequences.

For convenience we fix the following notation.
Let $(A;B=\al A;C;D)\in NN(n)$, $n=2m$, and let
$$
p_1p_2\ldots p_mp_{m+1} \quad \text{resp.} \quad q_1q_2\ldots q_m
$$
be the encoding of the pair $(A;B)$ resp. $(C;D)$.

\begin{definition}
We say that the near-normal sequences $(A;B;C;D)$
are in the {\em canonical form} if the following conditions hold:

(i) $p_1=0$, $q_1=1$.

(ii) If $q_j=2$ for some $j$, then $q_i=7$ for some index $i$
with $1<i<j$.

(iii)  If $q_j\in\{3,4,5\}$ for some $j$, then $q_i=6$ for some index
$i$ with $1<i<j$.

(iv) If $q_k\ne7$ for all $k$'s and $q_j=4$ for some $j$,
then $q_i=5$ for some index $i$ with $1<i<j$.
\end{definition}

The following proposition shows how one can enumerate the $BS$-equivalence
classes of $NN(n)$.

\begin{proposition} \label{klase}
For each $BS$-equivalence class $\pE\subseteq NN(n)$, $n=2m$, there is a unique
$(A;B;C;D)\in\pE$ having the canonical form.
\end{proposition}
\begin{proof}
Let $(A;B;C;D)\in\pE$ be arbitrary and let
$p_1p_2\ldots p_mp_{m+1}$ resp.  $q_1q_2\ldots q_m$
be the encoding of the pair $(A;B)$ resp. $(C;D)$.
By applying the first three
types of elementary transformations we can assume that $p_1=0$
and $c_1=d_1=+1$. Then $q_1$ must be either 1 or 6. In the latter case
we apply the elementary transformation (iv). Thus we may assume that
$p_1=0$ and $q_1=1$, i.e., the condition (i) for the canonical form
is satisfied.

Now assume that $q_j=2$ for some $j$ and that $q_i\ne7$
for all $i<j$. After interchanging the sequences $C$ and $D$, we obtain that
$q_j=7$ and $q_i\ne2$ for $i<j$. Hence we may also assume that the
condition (ii) is satisfied.

Next assume that $q_j\in\{3,4,5\}$ for some $j$.
We may take $j$ to be minimal with this property.
Assume that $q_i\ne6$ for $i<j$.
Consequently, $q_i\in\{1,2,7,8\}$ for all $i<j$. 
If $q_j=3$ we replace $(C;D)$ with
$(C';D')$. If $q_j=4$ we replace $D$ with $D'$. If $q_j=5$ we replace
$C$ with $C'$. After this change, we obtain that in all three cases
$q_j=6$ while the $q_i$'s for $i<j$ remain unchanged. Hence the
condition (iii) is also satisfied.

Finally, assume that $q_k\ne7$ for all $k$'s, $q_j=4$ for some $j$,
and $q_i\ne5$ for $i<j$. Since the condition (ii) holds, we have
$q_i\in\{1,3,6,8\}$ for all $i<j$. After interchanging $C$ and $D$,
we obtain that $q_j=5$ while the $q_i$ with $i<j$ remain unchanged.
Hence now $(A;B;C;D)$ is in the canonical form.

It remains to prove the uniqueness assertion. Let
$$
S^{(k)}=(A^{(k)};B^{(k)};C^{(k)};D^{(k)})\in\pE,\quad (k=1,2)
$$
be in the canonical form. By Proposition \ref{orbite}, there exists
$g\in G$ such that $gS^{(1)}=S^{(2)}$.
Let $p_1^{(1)}p_2^{(1)}\ldots p_{m+1}^{(1)}$ resp.
$p_1^{(2)}p_2^{(2)}\ldots p_{m+1}^{(2)}$ be the encoding of the pair
$(A^{(1)};B^{(1)})$ resp. $(A^{(2)};B^{(2)})$.
Let $q_1^{(1)}q_2^{(1)}\ldots q_m^{(1)}$ resp.
$q_1^{(2)}q_2^{(2)}\ldots q_m^{(2)}$ be the encoding of the pair
$(C^{(1)};D^{(1)})$ resp. $(C^{(2)};D^{(2)})$.
Since $q_1^{(1)}=q_1^{(2)}=1$, $g$ must be in $H$.
Note that $H=H_1\times H_2$, where the subgroup $H_1$ resp. $H_2$
is generated by $\{\ve_1,\ve_2,\vf_1,\vf_2,\sigma_1\}$ resp.
$\{\ve_3,\ve_4,\vf_3,\vf_4,\sigma_2\}$.
Thus we have $g=h_1h_2$ with $h_1\in H_1$ and $h_2\in H_2$.
Consequently, $h_1\cdot(A^{(1)};B^{(1)})=(A^{(2)};B^{(2)})$ and
$h_2\cdot(C^{(1)};D^{(1)})=(C^{(2)};D^{(2)})$.
We also have the direct decomposition $E=E_1\times E_2$, where
$E_1=E\cap H_1$ and $E_2=E\cap H_2$.

Since $p_1^{(1)}=p_1^{(2)}=0$, the equality
$h_1\cdot(A^{(1)};B^{(1)})=(A^{(2)};B^{(2)})$ implies that $h_1\in E_1$. 
Thus $h_1=\ve_1^{e_1} \ve_2^{e_2} \vf_1^{f_1} \vf_2^{f_2}$ with
$e_1,e_2,f_1,f_2\in\{0,1\}$.
Since the first and the last terms of the sequences $A^{(1)}$
and $A^{(2)}$ are $+1$, we have $e_1=0$. It follows that the middle
terms of these two sequences are the same. As $S^{(1)}$ and $S^{(2)}$
are near-normal sequences, the sequences $B^{(1)}$ and $B^{(2)}$
must also have the same middle terms. Consequently, $e_2=0$.
Since the sequences $B^{(1)}$ and $B^{(2)}$ have the same first term,
$+1$, and the same last term, $-1$, we infer that $f_2=0$.
Consequently, $B^{(1)}=B^{(2)}$. As $A^{(1)}=\al B^{(1)}$ and
$A^{(2)}=\al B^{(2)}$, we infer that also $A^{(1)}=A^{(2)}$.

Since $q_1^{(1)}=q_1^{(2)}=1$, the equality
$h_2\cdot(C^{(1)};D^{(1)})=(C^{(2)};D^{(2)})$ implies that $h_2$ 
belongs to the subgroup of $H_2$ generated by $\{\vf_3,\vf_4,\sigma_2\}$.
Thus $h_2=\vf_3^{f_3} \vf_4^{f_4} \sigma_2^s$ with $f_3,f_4,s\in\{0,1\}$.

Assume that $q_j^{(1)}=7$ for some $j$. Choose the smallest such $j$.
Then $q_i^{(1)}\ne2$ for $1<i<j$.
The quads 1,3,6,8 are fixed by $\sigma_2$ and the quads 2,4,5,7 are permuted
via the involution $(2,7)(4,5)$. On the other hand,
the generators $\vf_3$ and $\vf_4$ fix the quads 1,2,7,8.
Since $S^{(1)}$ and $S^{(2)}$ are in  the canonical form it follows that
$s=0$ and so $q_j^{(2)}=7$ and $q_i^{(2)}\ne2$ for $1<i<j$. The equality
$\vf_3^{f_3} \vf_4^{f_4}\cdot(C^{(1)};D^{(1)})=(C^{(2)};D^{(2)})$ now
implies that $C^{(1)}=C^{(2)}$ and $D^{(1)}=D^{(2)}$.
Hence $S^{(1)}=S^{(2)}$.

The argument is similar if $q_j^{(1)}\ne7$ for all $j$, which implies
that also $q_j^{(1)}\ne2$ for all $j$.
\end{proof}

We proceed to define the $NN$-equivalence relation in $NN(n)$.
For this we need to introduce the {\em $NN$-elementary
transformations:}

(i) Negate both sequences $A;B$ or one of $C;D$.

(ii) Reverse one of the sequences $C;D$.

(iii) Interchange the sequences $A;B$ or $C;D$.

(iv) Replace the sequences $(A;B=\al A)$ with
$(\hat{A};\hat{B}=\al\hat{A})$ where
$$ \hat{A}=a_{n-1},a_2,a_{n-3},a_4,\ldots,a_1,a_n,a_{n+1}. $$

(v) Replace the sequences $(C;D)$ with the pair
$(\tilde{C};\tilde{D})$ which is defined by its encoding
$\tilde{q}_1\tilde{q}_2\ldots \tilde{q}_m$ with
$$
\tilde{q}_i=\begin{cases} 5, & \text{if $q_i=4$;} \\ 4,
& \text{if $q_i=5$;} \\ q_i & \text{otherwise.} \end{cases}
$$

(vi) Alternate all four sequences $A;B;C;D$.

\begin{lemma} \label{L1}
By using the above notation, we have
$N(\hat{A})+N(\hat{B})=N(A)+N(B)$
and $N(\tilde{C})+N(\tilde{D})=N(C)+N(D)$.
Consequently, the quadruples $(\hat{A};\hat{B};C;D)$
and $(A;B;\tilde{C};\tilde{D})$ belong to $NN(n)$.
\end{lemma}
\begin{proof}
We sketch the proof only for the first quadruple.
It suffices to show that for even $i$ and odd $j<n$ we have
$$ \hat{a}_i\hat{a}_j+\hat{b}_i\hat{b}_j=a_i a_j+b_i b_j. $$
This is indeed true since
$\hat{b}_j=\hat{a}_j$ and $\hat{a}_i+\hat{b}_i=0$.
\end{proof}

We say that two members of $NN(n)$ are {\em $NN$-equivalent} if one
can be transformed to the other by applying a finite sequence
of the $NN$-elementary transformations (i-vi).

Our main objective is to enumerate the $NN$-equivalence classes of
$NN(n)$ for even integers $n\le30$.

\section{Equivalence classes of near-normal sequences} \label{tab}

In Table 1 we list the codes for the representatives of the
$NN$-equivalence classes of $NN(n)$ for even $n\le30$.
All representatives are chosen in the canonical form.

\begin{center}
\begin{tabular}{|r|l|l|r|r|}
\multicolumn{5}{c}{Table 1: $NN$-equivalence classes of $NN(n)$} \\ \hline 
\multicolumn{1}{|c|}{} & \multicolumn{1}{c|}{$A$ \& $B$} & 
\multicolumn{1}{c|}{$C$ \& $D$} & \multicolumn{1}{c|}{$a,b,c,d$} &
\multicolumn{1}{c|}{ $a^*,b^*,c^*,d^*$ }\\ \hline
\multicolumn{5}{c}{ $n=2$ }\\ \hline
1 & 02 & 1 & $1,1,2,2$ & $3,-1,0,0$ \\ \hline
\multicolumn{5}{c}{ $n=4$ }\\ \hline
1 & 050 & 16 & $3,1,2,2$ & $3,1,-2,-2$ \\
2 & 073 & 17 & $-1,1,0,4$ & $3,-3,0,0$ \\ \hline
\multicolumn{5}{c}{ $n=6$ }\\ \hline
1 & 0711 & 188 & $3,3,-2,-2$ & $5,1,0,0$ \\
2 & 0512 & 172 & $3,3,2,2$ & $5,1,0,0$ \\ \hline
\multicolumn{5}{c}{ $n=8$ }\\ \hline
1 & 07643 & 1651 & $-1,1,4,4$ & $3,-3,-4,0$ \\
2 & 05850 & 1163 & $1,-1,4,4$ & $1,-1,4,4$ \\
3 & 05323 & 1637 & $3,-3,0,4$ & $-1,1,-4,-4$ \\ \hline
\multicolumn{5}{c}{ $n=10$ }\\ \hline
1 & 076462 & 16712 & $-1,3,4,4$ & $5,-3,-2,-2$ \\
2 & 078211 & 16561 & $3,-1,4,4$ & $1,1,-6,-2$ \\
3 & 078412 & 16787 & $-1,3,-4,4$ & $5,-3,-2,-2$ \\
4 & 076761 & 17621 & $-1,3,4,4$ & $5,-3,2,2$ \\
5 & 056732 & 11726 & $-1,3,4,4$ & $5,-3,2,2$ \\
6 & 058511 & 11635 & $3,-1,4,4$ & $1,1,2,6$ \\
7 & 053281 & 16355 & $3,-5,2,2$ & $-3,1,-4,-4$ \\
8 & 053781 & 17616 & $-1,-1,2,6$ & $1,-3,4,4$ \\ \hline 
\multicolumn{5}{c}{ $n=12$ }\\ \hline
1 & 0765373 & 161762 & $-3,3,4,4$ & $5,-5,0,0$ \\ 
2 & 0764373 & 165513 & $-3,3,4,4 $ & $5,-5,0,0$ \\
3 & 0764320 & 165776 & $3,1,-2,6 $ & $3,1,-6,-2$ \\
4 & 0764870 & 167162 & $-3,3,4,4 $ & $5,-5,0,0$ \\
5 & 0784370 & 167867 & $-3,3,-4,4 $ & $5,-5,0,0$ \\
6 & 0765373 & 176738 & $-3,3,-4,4 $ & $5,-5,0,0$ \\
7 & 0715873 & 187766 & $-3,3,-4,4 $ & $5,-5,0,0$ \\
8 & 0737653 & 187222 & $-3,3,4,-4 $ & $5,-5,0,0$ \\
9 & 0585140 & 116754 & $3,1,2,6 $ & $3,1,-2,6$ \\
10 & 0517820 & 161675 & $3,1,2,6 $ & $3,1,-2,-6$ \\
11 & 0512673 & 165714 & $3,1,2,6 $ & $3,1,-6,2$ \\
12 & 0512870 & 167575 & $3,1,-2,6 $ & $3,1,2,-6$ \\
13 & 0515643 & 176153 & $3,1,2,6 $ & $3,1,2,6$ \\
14 & 0515343 & 176547 & $3,1,-2,6 $ & $3,1,6,-2$ \\
\hline 
\end{tabular} \\
\end{center}

\begin{center}
\begin{tabular}{|r|l|l|r|r|}
\multicolumn{5}{c}{Table 1 (continued)} \\ \hline 
\multicolumn{1}{|c|}{} & \multicolumn{1}{c|}{$A$ \& $B$} & 
\multicolumn{1}{c|}{$C$ \& $D$} & \multicolumn{1}{c|}{$a,b,c,d$} &
\multicolumn{1}{c|}{ $a^*,b^*,c^*,d^*$ }\\ \hline
\multicolumn{5}{c}{ $n=14$ }\\ \hline
1 & 07623211 & 1637668 & $7,-1,-2,2$ & $1,5,-4,-4$ \\ 
2 & 07621231 & 1651468 & $7,-1,2,2 $ & $1,5,-4,-4$ \\
3 & 07643511 & 1675657 & $3,3,-2,6 $ & $5,1,4,-4$ \\
4 & 07676212 & 1763321 & $1,5,4,4 $ & $7,-1,2,2$ \\
5 & 07176262 & 1868866 & $1,5,-4,-4 $ & $7,-1,2,2$ \\
6 & 07378282 & 1865311 & $-5,-1,4,4 $ & $1,-7,2,-2$ \\
7 & 05673512 & 1172663 & $1,5,4,4 $ & $7,-1,-2,-2$ \\
8 & 05821712 & 1187763 & $3,3,-2,6 $ & $5,1,-4,-4$ \\
9 & 05128712 & 1638177 & $3,3,-2,6 $ & $5,1,-4,-4$ \\
10 & 05121562 & 1678524 & $7,3,0,0 $ & $5,5,-2,-2$ \\
11 & 05146762 & 1678376 & $1,5,-4,4 $ & $7,-1,-2,-2$ \\ \hline 
\multicolumn{5}{c}{ $n=16$ }\\ \hline
1 & 076567350 & 16117872 & $-1,5,2,6$ & $7,-3,-2,-2$ \\ 
2 & 076215320 & 16333817 & $7,1,0,4$ & $3,5,-4,-4$ \\ 
3 & 076212650 & 16373355 & $7,1,0,4$ & $3,5,-4,-4$ \\ 
4 & 076214670 & 16377568 & $3,5,-4,4$ & $7,1,0,-4$ \\ 
5 & 076487150 & 16716223 & $-1,5,6,2$ & $7,-3,2,2$ \\ 
6 & 076417643 & 16752321 & $-1,5,6,2$ & $7,-3,2,-2$ \\ 
7 & 076417343 & 16756467 & $-1,5,-2,6$ & $7,-3,2,2$ \\ 
8 & 076715643 & 17265377 & $-1,5,-2,6$ & $7,-3,-2,2$ \\ 
9 & 076517353 & 17661518 & $-1,5,2,6$ & $7,-3,2,-2$ \\ 
10 & 076534120 & 17665214 & $5,3,4,4$ & $5,3,-4,4$ \\ 
11 & 076587150 & 17766813 & $-1,5,-2,6$ & $7,-3,2,2$ \\ 
12 & 076487150 & 17816788 & $-1,5,-6,2$ & $7,-3,2,2$ \\ 
13 & 071564320 & 18676557 & $5,3,-4,4$ & $5,3,4,4$ \\ 
14 & 071265620 & 18863557 & $7,1,-4,0$ & $3,5,-4,-4$ \\ 
15 & 051284823 & 16546732 & $3,-7,2,2$ & $-5,1,-6,2$ \\ 
16 & 051235623 & 16637385 & $7,-3,-2,2$ & $-1,5,6,2$ \\ 
17 & 051462323 & 16654553 & $7,-3,2,2$ & $-1,5,6,-2$ \\ 
18 & 051267640 & 16753874 & $5,3,-4,4$ & $5,3,-4,-4$ \\ 
19 & 053467670 & 16537515 & $-1,5,2,6$ & $7,-3,2,-2$ \\ 
20 & 053467873 & 16754414 & $-5,1,2,6$ & $3,-7,-2,-2$ \\ 
21 & 053462823 & 16758534 & $3,-7,-2,2$ & $-5,1,-2,-6$ \\ 
22 & 051712820 & 17268876 & $7,1,-4,0$ & $3,5,-4,-4$ \\ 
23 & 051564173 & 17726215 & $3,5,4,4$ & $7,1,4,0$ \\ 
24 & 051467373 & 17886836 & $-1,5,-6,-2$ & $7,-3,-2,-2$ \\ 
\hline
\end{tabular} \\
\end{center}

\begin{center}
\begin{tabular}{|r|l|l|r|r|}
\multicolumn{5}{c}{Table 1 (continued)} \\ \hline 
\multicolumn{1}{|c|}{} & \multicolumn{1}{c|}{$A$ \& $B$} & 
\multicolumn{1}{c|}{$C$ \& $D$} & \multicolumn{1}{c|}{$a,b,c,d$} &
\multicolumn{1}{c|}{ $a^*,b^*,c^*,d^*$ }\\ \hline
\multicolumn{5}{c}{ $n=18$ }\\ \hline
1 & 0767653462 & 161544647 & $-3,5,2,6$ & $7,-5,0,0$ \\
2 & 0762328231 & 163544668 & $5,-7,0,0$ & $-5,3,-2,-6$ \\
3 & 0762328211 & 163554338 & $7,-5,0,0$ & $-3,5,-6,-2$ \\
4 & 0762328211 & 165138835 & $7,-5,0,0$ & $-3,5,-6,2$ \\
5 & 0782156561 & 165413577 & $3,-1,0,8$ & $1,1,-6,6$ \\
6 & 0782143782 & 165726177 & $-3,1,0,8$ & $3,-5,-6,-2$ \\
7 & 0767178262 & 172782221 & $-3,5,6,-2$ & $7,-5,0,0$ \\
8 & 0767643462 & 176672155 & $-3,5,2,6$ & $7,-5,0,0$ \\
9 & 0765153782 & 177821181 & $-3,5,2,6$ & $7,-5,0,0$ \\
10 & 0785153762 & 172212188 & $-3,5,6,-2$ & $7,-5,0,0$ \\
11 & 0737846761 & 186557167 & $-5,3,-2,6$ & $5,-7,0,0$ \\
12 & 0567156482 & 117763815 & $-1,3,0,8$ & $5,-3,2,6$ \\
13 & 0512876462 & 165351136 & $1,1,6,6$ & $3,-1,0,8$ \\
14 & 0514846732 & 166751736 & $-1,3,0,8$ & $5,-3,2,6$ \\
15 & 0512846282 & 167654744 & $3,-5,-2,6$ & $-3,1,-8,0$ \\
16 & 0512846262 & 167655745 & $5,-3,-2,6$ & $-1,3,-8,0$ \\
17 & 0512848232 & 167655745 & $3,-5,-2,6$ & $-3,1,-8,0$ \\
18 & 0532376482 & 165351136 & $-1,-1,6,6$ & $1,-3,0,8$ \\
19 & 0517846761 & 176567164 & $-1,3,0,8$ & $5,-3,6,-2$ \\
20 & 0537346781 & 176567164 & $-3,1,0,8$ & $3,-5,6,-2$ \\ \hline
\multicolumn{5}{c}{ $n=20$ }\\ \hline
1 & $07621517870$ & $1633771868$ & $1,7,-4,4$ & $9,-1,0,0$\\
2 & $07643282143$ & $1657513537$ & $3,-3,0,8$ & $-1,1,-8,-4$\\
3 & $07643215823$ & $1657761672$ & $3,-3,0,8$ & $-1,1,-8,-4$\\
4 & $07643285123$ & $1676715655$ & $3,-3,0,8$ & $-1,1,-8,-4$\\
5 & $07821417670$ & $1655337213$ & $1,7,4,4$ & $9,-1,0,0$\\
6 & $07821464623$ & $1657367551$ & $3,-3,0,8$ & $-1,1,-8,-4$\\
7 & $07156514620$ & $1876332551$ & $7,5,2,2$ & $7,5,-2,-2$\\
8 & $07356484873$ & $1871628611$ & $-7,-1,4,4$ & $1,-9,0,0$\\
9 & $05673282320$ & $1166536724$ & $5,-5,4,4$ & $-3,3,0,8$\\
10 & $05153467820$ & $1616571625$ & $3,1,6,6$ & $3,1,-6,-6$\\
11 & $05178262840$ & $1616372252$ & $3,-3,8,0$ & $-1,1,-8,-4$\\
12 & $05146784840$ & $1663611547$ & $-1,1,4,8$ & $3,-3,8,0$\\
13 & $05146214173$ & $1665572814$ & $7,5,2,2$ & $7,5,-2,2$\\
14 & $05146515153$ & $1678325325$ & $7,5,2,-2$ & $7,5,-2,-2$\\
15 & $05146265620$ & $1678813524$ & $9,-1,0,0$ & $1,7,-4,-4$\\
16 & $05346265873$ & $1661754125$ & $-1,-3,6,6$ & $-1,-3,6,-6$\\
17 & $05171564620$ & $1726655445$ & $7,5,2,2$ & $7,5,2,-2$\\
18 & $05126532340$ & $1786556323$ & $9,-1,0,0$ & $1,7,4,4$\\
\hline
\end{tabular} \\
\end{center}

\begin{center}
\begin{tabular}{|r|l|l|r|r|}
\multicolumn{5}{c}{Table 1 (continued)} \\ \hline 
\multicolumn{1}{|c|}{} & \multicolumn{1}{c|}{$A$ \& $B$} & 
\multicolumn{1}{c|}{$C$ \& $D$} & \multicolumn{1}{c|}{$a,b,c,d$} &
\multicolumn{1}{c|}{ $a^*,b^*,c^*,d^*$ }\\ \hline
\multicolumn{5}{c}{ $n=22$ }\\ \hline
1 & $076537321212$ & $16156871224$ & $5,5,6,2$ & $7,3,4,-4$ \\
2 & $076212641431$ & $16353377225$ & $9,1,2,2$ & $3,7,-4,-4$ \\
3 & $076487121512$ & $16337381132$ & $3,7,4,4$ & $9,1,2,2$ \\
4 & $076414343562$ & $16557178513$ & $1,5,0,8$ & $7,-1,-6,-2$ \\
5 & $076414343562$ & $16561764357$ & $1,5,0,8$ & $7,-1,-6,-2$ \\
6 & $076414378212$ & $16617767256$ & $1,5,0,8$ & $7,-1,6,2$ \\
7 & $076435641411$ & $16761544847$ & $5,5,-2,6$ & $7,3,-4,-4$ \\
8 & $078212153261$ & $16778255254$ & $9,-3,0,0$ & $-1,7,2,-6$ \\
9 & $078435173511$ & $16787588663$ & $1,5,-8,0$ & $7,-1,-2,-6$ \\
10 & $076512676432$ & $17633151578$ & $1,5,0,8$ & $7,-1,-2,6$ \\
11 & $076515671481$ & $17675144618$ & $1,5,0,8$ & $7,-1,2,6$ \\
12 & $076782121711$ & $17652175378$ & $5,5,-2,6$ & $7,3,4,-4$ \\
13 & $071567356562$ & $18767883255$ & $-1,7,-6,-2$ & $9,-3,0,0$ \\
14 & $071584328782$ & $18768533758$ & $-5,-1,-8,0$ & $1,-7,2,-6$ \\
15 & $056414173761$ & $11868766736$ & $3,7,-4,4$ & $9,1,2,2$ \\
16 & $058512141532$ & $11635676523$ & $7,3,4,4$ & $5,5,-6,2$ \\
17 & $058512328781$ & $11637254662$ & $1,-7,6,2$ & $-5,-1,0,8$ \\
18 & $051715853212$ & $16187155327$ & $5,5,2,6$ & $7,3,-4,-4$ \\
19 & $051265126462$ & $16534782626$ & $9,1,2,-2$ & $3,7,4,4$ \\
20 & $051265128432$ & $16535712626$ & $7,-1,6,2$ & $1,5,0,8$ \\
21 & $051284651712$ & $16576127148$ & $5,5,2,6$ & $7,3,-4,4$ \\
22 & $051234648212$ & $16654867176$ & $7,-1,-2,6$ & $1,5,8,0$ \\
23 & $051464641232$ & $16675487723$ & $7,3,-4,4$ & $5,5,-6,2$ \\
24 & $051265673412$ & $16723155718$ & $5,5,2,6$ & $7,3,-4,-4$ \\
25 & $053482826781$ & $16383573582$ & $-1,-9,-2,-2$ & $-7,-3,-4,-4$ \\
26 & $053465151281$ & $16537353721$ & $7,-1,2,6$ & $1,5,0,8$ \\
27 & $053265626512$ & $16712758341$ & $7,-1,2,6$ & $1,5,-8,0$ \\
28 & $053464621711$ & $16728657537$ & $7,3,-4,4$ & $5,5,-6,2$ \\
29 & $051515841782$ & $17631554474$ & $1,5,0,8$ & $7,-1,6,2$ \\
30 & $051512658432$ & $17653363147$ & $5,1,0,8$ & $3,3,6,6$ \\
31 & $051762648211$ & $17657861418$ & $7,-1,-2,6$ & $1,5,8,0$ \\
32 & $051567326261$ & $17763546681$ & $7,-1,-2,6$ & $1,5,0,-8$ \\
\hline
\end{tabular} \\
\end{center}

\begin{center}
\begin{tabular}{|r|l|l|r|r|}
\multicolumn{5}{c}{Table 1 (continued)} \\ \hline 
\multicolumn{1}{|c|}{} & \multicolumn{1}{c|}{$A$ \& $B$} & 
\multicolumn{1}{c|}{$C$ \& $D$} & \multicolumn{1}{c|}{$a,b,c,d$} &
\multicolumn{1}{c|}{ $a^*,b^*,c^*,d^*$ }\\ \hline
\multicolumn{5}{c}{ $n=24$ }\\ \hline
1 & $0765321785873$ & $161653745512$ & $-5,1,6,6$ & $3,-7,-6,2$ \\
2 & $0764156484370$ & $165371748678$ & $-1,5,-6,6$ & $7,-3,-6,2$ \\
3 & $0767621532650$ & $176577612445$ & $5,3,0,8$ & $5,3,8,0$ \\
4 & $0767328512140$ & $176833835874$ & $5,3,-8,0$ & $5,3,0,8$ \\
5 & $0715653785350$ & $187677653446$ & $-1,5,-6,6$ & $7,-3,-2,-6$ \\
6 & $0737626512340$ & $186537753131$ & $5,3,0,8$ & $5,3,0,-8$ \\
7 & $0734876235823$ & $188663535787$ & $-3,-5,-8,0$ & $-3,-5,0,8$ \\
8 & $0512653237623$ & $165353436747$ & $7,-3,-2,6$ & $-1,5,6,6$ \\
9 & $0532651262673$ & $167167854247$ & $7,-3,-2,6$ & $-1,5,6,-6$ \\
10 & $0515178265340$ & $176765741452$ & $5,3,0,8$ & $5,3,0,8$ \\
11 & $0517646732123$ & $176654163667$ & $5,3,0,8$ & $5,3,-8,0$ \\
12 & $0512158564370$ & $178868726547$ & $5,3,-8,0$ & $5,3,8,0$ \\ \hline
\multicolumn{5}{c}{ $n=26$ }\\ \hline
1 & $07641487843412$ & $1654611475266$ & $-3,5,6,6$ & $7,-5,-4,4$ \\
2 & $05128264656262$ & $1654776852733$ & $7,-5,-4,4$ & $-3,5,-6,6$ \\
3 & $05126265841481$ & $1782657541321$ & $7,-5,4,4$ & $-3,5,6,-6$ \\ \hline
\multicolumn{5}{c}{ $n=28$ }\\ \hline
1 & $076534321432170$ & $16178852836758$ & $7,5,-6,-2$ & $7,5,-2,6$ \\
2 & $076232648787870$ & $16354457772331$ & $-7,-1,0,8$ & $1,-9,-4,-4$ \\
3 & $078232123565140$ & $16538735377542$ & $9,-1,-4,4$ & $1,7,0,-8$ \\
4 & $078215348487673$ & $16754388724478$ & $-7,-1,-8,0$ & $1,-9,4,-4$ \\
5 & $078214148264370$ & $16767651613356$ & $3,1,2,10$ & $3,1,-10,-2$ \\
6 & $076512326587853$ & $17635447785113$ & $-1,-3,-2,10$ & $-1,-3,2,10$ \\
7 & $076514146435673$ & $17655216547871$ & $1,7,0,8$ & $9,-1,-4,4$ \\
8 & $076537321737843$ & $17661856774521$ & $-5,5,0,8$ & $7,-7,0,-4$ \\
9 & $076582151735173$ & $17727186654441$ & $1,7,0,8$ & $9,-1,4,-4$ \\
10 & $078517356737323$ & $17262157855212$ & $-5,5,8,0$ & $7,-7,-4,0$ \\
11 & $078235123464120$ & $17863835255348$ & $9,-1,-4,-4$ & $1,7,0,-8$ \\
12 & $071564621714873$ & $18637255448877$ & $1,7,-8,0$ & $9,-1,4,4$ \\
13 & $071564148764650$ & $18667117672553$ & $1,7,0,8$ & $9,-1,4,4$ \\
14 & $071287651232320$ & $18876654763441$ & $9,-1,-4,4$ & $1,7,-8,0$ \\
15 & $053465153484843$ & $16754378876583$ & $-1,-3,-10,2$ & $-1,-3,10,-2$ \\
16 & $051765146467353$ & $17631822512665$ & $1,7,8,0$ & $9,-1,4,4$ \\
17 & $051762846767140$ & $17654781165581$ & $1,7,0,8$ & $9,-1,4,-4$ \\
18 & $051784828462343$ & $17656316516487$ & $-1,-7,0,8$ & $-5,-3,4,8$ \\
19 & $051782353215153$ & $17678365277211$ & $7,1,0,8$ & $3,5,8,4$ \\
20 & $051567121285343$ & $17765468271156$ & $7,1,0,8$ & $3,5,-8,4$ \\
\hline
\end{tabular} \\
\end{center}

\begin{center}
\begin{tabular}{|r|l|l|r|r|}
\multicolumn{5}{c}{Table 1 (continued)} \\ \hline 
\multicolumn{1}{|c|}{} & \multicolumn{1}{c|}{$A$ \& $B$} & 
\multicolumn{1}{c|}{$C$ \& $D$} & \multicolumn{1}{c|}{$a,b,c,d$} &
\multicolumn{1}{c|}{ $a^*,b^*,c^*,d^*$ }\\ \hline
\multicolumn{5}{c}{ $n=30$ }\\ \hline
1  & $0782321435141431$ & $167587656743842$ & $9,1,-6,2$ & $3,7,8,0$ \\
2  & $0784151482828782$ & $167835857653471$ & $-5,-5,-6,6$ & $-3,-7,0,-8$ \\
3  & $0784351765121731$ & $167838232233854$ & $3,7,0,-8$ & $9,1,2,-6$ \\
4  & $0767641512178561$ & $176536611456768$ & $3,7,0,8$ & $9,1,-6,-2$ \\
5  & $0564376515151581$ & $118772615545132$ & $5,5,6,6$ & $7,3,8,0$ \\
6 & $0512656235371531$ & $165711846213678$ & $9,1,2,6$ & $3,7,0,8$ \\
7 & $0534678534348481$ & $165344387727573$ & $-5,-5,-6,6$ & $-3,-7,-8,0$ \\
8 & $0532678482348461$ & $167165812256464$ & $-1,-9,6,2$ & $-7,-3,0,-8$ \\
9 & $0515153564821232$ & $177863718512664$ & $9,1,-2,6$ & $3,7,8,0$ \\
\hline
\end{tabular} \\
\end{center}


\begin{thebibliography}{99}

\bibitem{HCD}
C.J. Colbourn and J.H. Dinitz, Editors, Handbook of Combinatorial Designs,
2nd edition, Chapman \& Hall, Boca Raton/London/New York, 2007.

\bibitem{DZ1}
D.\v{Z}. \DJo{},
Aperiodic complementary quadruples of binary sequences,
JCMCC {\bf 27} (1998), 3--31. Correction: ibid {\bf 30} (1999), p. 254.

\bibitem{DZ2}
\bysame, Some new near-normal sequences (preprint).

\bibitem{KSY}
C. Koukouvinos, S. Kounias, J. Seberry, C.H. Yang and J. Yang,
Multiplication of sequences with zero autocorrelation,
Austral. J. Combin. {\bf 10} (1994), 5--15.

\bibitem{SY}
J. Seberry and M. Yamada, Hadamard matrices, sequences and block
designs, in Contemporary Design Theory: A Collection of Surveys,
Eds. J.H. Dinitz and D.R. Stinson, J. Wiley, New York, 1992,
pp. 431--560.

\bibitem{Y}
C.H. Yang, On composition of four-symbol $\delta$-codes and
Hadamard matrices, Proc. Amer. Math. Soc. {\bf 107} (1989), 763--776.

\end{thebibliography}
\end{document}